\documentclass[11pt]{amsart}
\usepackage{amssymb, amsmath}
\usepackage{amsthm}
\usepackage{amscd}
\usepackage{graphicx}
\newtheorem{theorem}{Theorem}
\newtheorem{lemma}[theorem]{Lemma}

\newtheorem{question}[theorem]{Question}
\newtheorem{corollary}[theorem]{Corollary}
\newtheorem{proposition}[theorem]{Proposition}
\newtheorem{conjecture}[theorem]{Conjecture}
\theoremstyle{definition}

\newtheorem*{example}{Example}
\newtheorem*{remark}{Remark}

\newtheorem{definition}[theorem]{Definition}
\title{Rational manifold models for duality groups}
\author{Grigori Avramidi}
\address{Mathematics Institute\\ 
        University of M\"unster\\ 
        Germany}
\email{avramidi@uni-muenster.de}

\def\ra{\rightarrow}

\def\beqa{\begin{eqnarray}}
\def\eeqa{\end{eqnarray}}
\def\beqa{\begin{eqnarray}}
\def\eeqa{\end{eqnarray}}

\DeclareMathOperator{\im}{Im}

\DeclareMathOperator{\Out}{Out}

\DeclareMathOperator{\SL}{SL}

\input xy
\xyoption{all}
\begin{document}

\begin{abstract}
We show that a finite type duality group of dimension $d>2$ is the fundamental group of a $(d+3)$-manifold with rationally acyclic universal cover. We use this to find closed manifolds with rationally acyclic universal cover and some nonvanishing $L^2$-Betti numbers outside the middle dimension, which contradicts a rational analogue of a conjecture of Singer. 
\end{abstract}
\maketitle

\section*{Introduction}
A classifying space for a discrete group $\Gamma$ is any complex $B\Gamma$ with fundamental group $\Gamma$ and contractible universal cover. If a group has a $d$-dimensional classifying space then it always has a $2d$-dimensoinal {\it manifold} classifying space \cite{stallingsembedding}. This motivates the following question. 

\begin{question}
\label{acquestion}
What is the smallest dimension of a manifold $B\Gamma$?
\end{question}
This question has been studied for many groups 
and answered for some of them \cite{avramididavisokunschreve,bestvinafeighn, bestvinakapovichkleiner,despotovic}. A crucial example to keep in mind is the Cartesian product of $d$ copies of the free group on two generators $F_2^d$. It is shown in \cite{bestvinakapovichkleiner} that any manifold $BF_2^d$ has dimension $\geq 2d$. 

It has recently been shown \cite{okunschreve} that the existence of sufficiently low dimensional manifold classifying spaces for certain groups would lead to counterexamples to Singer's vanishing conjecture for $L^2$-Betti numbers of closed aspherical manifolds  (see subsection \ref{vanishingconjectures}). However, for the classical groups (finite index torsionfree subgroups of lattices in locally symmetric spaces, mapping class groups, outer automorphism groups of free groups etc.) one knows that such low dimensional manifold classifying spaces do not exist. 

On the other hand, since the early days of manifold theory it has been evident that things can be much more tractable when studied up to ``finite ambiguity'' (see the introduction of \cite{sullivan} for a compelling exposition of this philosophy).
In this spirit, this paper investigates a rational version of the above question. 
Two spaces $X$ and $Y$ are {\it rationally equivalent} if there is a $\pi_1$-isomorphism $X\ra Y$ that lifts to a rational homology isomorphism of universal covers $\widetilde X\ra\widetilde Y$.
\begin{question}
What is the smallest dimension of a manifold rationally equivalent to $B\Gamma$?
\end{question}
Our main result shows that the obstructions to obtaining low dimensional manifold classifying spaces can (and often do) disappear rationally. 
\begin{theorem}
\label{maintheorem}
Suppose that $d>2$ and $\Gamma$ is a $d$-dimensional duality group $\Gamma$ with finite classifying space $B\Gamma$. Then $B\Gamma$ is rationally equivalent to a $(d+3)$-manifold.
\end{theorem}

The manifolds constructed for this theorem are compact manifolds-with-boundary $(M,\partial)$ whose universal cover $\widetilde M$ is rationally acyclic. The reflection group method \cite{davisbook} can be applied (along the lines of \cite{okunschreve}) to these $(M,\partial)$ to get closed manifolds with $\mathbb Q$-acyclic universal cover and $L^2$-Betti numbers appearing above the middle dimension. 
\begin{remark}
A manifold has rationally acyclic universal cover if and only if all of its higher homotopy groups are torsion, so we will call such  manifolds {\it rationally aspherical}.
\end{remark}
\begin{theorem}
\label{closedtheorem}
There is a closed, rationally aspherical manifold (of dimension $\leq 7$) whose $L^2$-Betti numbers are not concentrated in the middle dimension. 
\end{theorem} 

\subsection*{Outline of Theorem \ref{maintheorem}: Equivariant Moore spaces and manifold classifying spaces }
For a $\Gamma$-module $D$ and an integer $r\geq 2$, a {\it $\Gamma$-equivariant Moore space $M(r,D)$} is any $\Gamma$-complex with (reduced) homology concentrated in dimension $r$ and equal to the $\Gamma$-module $D$. In section \ref{rationalhomotopysection} we 
construct rational versions of such $\Gamma$-equivariant Moore spaces.

For a $d$-dimensional, finite type duality group $\Gamma$ with dualizing module $D$, the existence of an equivariant Moore space $M(r,D)$ is closely related to the existence of a $(d+r+1)$-dimensional manifold classifying space. 
\begin{itemize}
\item
In one direction, if $M$ is a $(d+r+1)$-dimensional manifold classifying space and $M$ is the interior of a compact manifold with boundary $(M,\partial)$, then the universal cover of the boundary $\widetilde\partial$ is a $\Gamma$-equivariant Moore space $M(r,D)$ (see (\ref{duality})).  
\item
In the other direction, one starts with a $\Gamma$-equivariant Moore space and tries to realize it (up to homotopy) as the universal cover of the boundary of a manifold classifying space. In sections \ref{dualitysection} and \ref{surgerysection} we show that rationally this can be done. In more detail, section \ref{dualitysection} uses the fact that $D$ is the dualizing module of a duality group to show $M(r,D)$ has an appropriate Poincare duality ``built in'' and section \ref{surgerysection} uses this to build a $(d+r+1)$-manifold rationally equivalent to $B\Gamma$.  
\end{itemize}
Theorem \ref{maintheorem} follows from this and the existence of rational $\Gamma$-equivariant Moore spaces.

\subsection*{The kernel and spectrum of the Laplacian on forms on the universal cover}
We wish to discuss the relation of the results of this paper to questions about the kernel and spectrum of the Laplace operator on differential forms.

A theme connecting analysis and topology is that global analysis (elliptic operators, $L^2$-differential forms and all that stuff) on the universal cover of a closed manifold tells us about the topology of the base manifold (e.g. about the Euler characteristic, signature, and growth of Betti numbers in finite covers). When the universal cover is contractible, then the simplest possible analytic picture is that the $L^2$-deRham cohomology of the universal cover (represented by the harmonic $L^2$-forms) appears only in the middle dimension. This was first observed by Dodziuk and Singer in the context of rotationally symmetric universal covers \cite{dodziuk}, and they suggested that it might hold more generally. Since then, this picture has proved to be remarkably resilient. For instance, it is correct for locally symmetric manifolds \cite{olbrich}\footnote{The proof of this is chronologically earlier than the result for rotationally symmetric manifolds, so Dodziuk and Singer probably knew this.}, K\"ahler hyperbolic manifolds \cite{gromovkaehler}, even dimensional manifolds of sufficiently pinched negative curvature \cite{donnellyxavier}, manifolds whose fundamental groups contain an infinite normal amenable subgroup \cite{cheegergromov}, and four-dimensional right angled Coxeter manifolds \cite{davisokun}. 
The proofs of these results range from analytic (in the first three cases) to combinatorial (in the last one).

On the other hand, evidence against this picture is provided by Anderson's result \cite{andersoncounter} that without the cocompact $\pi_1$-action small perturbations of the hyperbolic metric on hyperbolic $3$-space admit $L^2$-harmonic $1$-forms. There is also the general theme (especially striking in the case of symmetric manifolds) that as one moves from strongly pinched negative curvature to non-positive curvature\footnote{With the torus as an extreme case.}, the spectrum of the Laplacian on forms can contain zero in a band around the middle dimension, even though genuine $L^2$-harmonic forms only appear in the middle dimension. These two things suggest that it is difficult to prove the Dodziuk-Singer picture in general using purely analytic methods. The results of this paper provide further support for this point of view by giving evidence of a different sort: this time involving a cocompact $\pi_1$-action and universal covers that are ``contractible from the point of view of deRham cohomology''.  

One could be led to these ideas by the following train of thought. The paper \cite{okunschreve} proved that the analytic Dodziuk-Singer picture, if true for all closed aspherical manifolds without any curvature assumptions\footnote{The original question in \cite{dodziuk} was for nonpositively curved manifolds. Gromov pointed out the extent of ignorance about this question by noting that for all we know it might be true for all closed aspherical manifolds, or there may be counterexamples even among strongly pinched branched covers of odd dimensional hyperbolic manifolds (\cite{gromovasymptotic} Section 8.$A_1$(G), page 152). It seems that the idea of extending the question to aspherical manifolds has been picked up, but without the implied skepticism.}, would have the remarkable topological consequence that $L^2$-Betti numbers are obstructions to the existence of low dimensional thickenings\footnote{A thickening of a  complex is a manifold homotopy equivalent to it.} of aspherical complexes\footnote{One can define $L^2$-Betti numbers more combinatorially for finite complexes. If the finite complex is a closed manifold, then this agrees with the more analytic definition (by \cite{dodziukderham}).} (the possibility of this was already suggested in \cite{davisokun}). Thickenings were studied in \cite{bestvinakapovichkleiner} where it was shown for a nonpositively curved complex $X$ that a van Kampen type embedding obstruction (a torsion obstruction) for $\partial_{\infty}\widetilde X$ obstructs thickenings of $X$.    
This was investigated further in the special case of (right angled) Salvetti complexes\footnote{Salvetti complexes are a certain class of nonpositively curved complexes generalizing a cartesian product $8\times\dots\times 8$ of several copies of the figure eight.}. The paper \cite{avramididavisokunschreve} showed (for $k\not=3$) that the van Kampen obstruction is the only obstruction to thickening a $k$-dimensional Salvetti complex $X$ into a $(2k-1)$-manifold. It also showed that this van Kampen obstruction doesn't vanish if $b_k^{(2)}(X)\not=0$. This suggests two things:
\begin{enumerate}
\item
\label{nocurvature}
The van Kampen embedding obstruction is torsion. 
Is this a general feature of thickening obstructions? 
\item
\label{npc}
For a finite nonpositively curved complex (or a closed nonpositively curved manifold) $X$, one should study the $L^2$-Betti numbers in terms of the boundary at infinity $\partial_{\infty}\widetilde X$ and try to relate them directly to embedding obstructions for $\partial_{\infty}\widetilde X$.   
\end{enumerate} 
From one point of view, Theorem \ref{maintheorem} addresses item (\ref{nocurvature}) by showing that thickening obstructions are torsion in a large range of dimensions and for all finite aspherical complexes whose fundamental groups are duality groups\footnote{It is an interesting question to determine whether the {\it duality group} assumption is truly necessary, or is merely a simplifying assumption.}. So, the torsion nature of the top dimensional thickening obstruction for Salvetti complexes is not an isolated phenomenon. 

Moreover, Theorem \ref{closedtheorem} shows that the Dodziuk-Singer question regarding (upper bounds on) the kernel of the Laplacian is of a different nature from the question whether the spectrum of the Laplacian contains zero\footnote{The conjecture that zero should be in the spectrum of the Laplacian on the universal cover of an aspherical manifold was formulated in \cite{gromovlarge} and the role of asphericity was clarified in \cite{farberweinberger}. In contrast to the Dodziuk-Singer question, the affirmative answer to it is known for a large class of groups because it is implied by the strong Novikov conjecture \cite{lott}.}. This is because proofs of the later (via the strong Novikov conjecture) work equally well for aspherical and rationally aspherical manifolds. So, the Dodziuk-Singer question is sensitive to the distinction between asphericity and rational asphericity, while the zero-in-the-spectrum question, in all the cases we are aware of, is not.
This suggests that the global analysis methods developed to study the later should not lead to a positive solution of the former without additional analytic (e.g. nonpositive curvature) assumptions which would a priori rule out the distinction between asphericity and rational asphericity.\footnote{This does not mean that some of the tools for studying the zero-in-the-spectrum question might not prove useful for the Dodziuk-Singer question. For instance, one might ask if twisting with representations of the fundamental group (in the spirit of \cite{luecktwisting}, but into groups that themselves possess $L^2$-Betti numbers in a range of dimensions instead of just $\mathbb Z^d$) can produce nonzero twisted $L^2$-Betti numbers outside the middle dimension.} 


By contrast, in the presence of (either combinatorial or analytic) nonpositive curvature, there are promising tools available to study the Dodziuk-Singer question, and in particular the relation suggested by item (\ref{npc}) between $L^2$-Betti numbers and embedding obstructions for the boundary at infinity. In fact, the idea of studying $L^2$-chains\footnote{And also $L^p$-chains for $p\not=2$.} on $\widetilde X$ in terms of the boundary at infinity\footnote{Or other similar boundaries.} $\partial_{\infty}\widetilde X$ appears in Section 8 of \cite{gromovasymptotic} (see also \cite{bourdonkleinercoxeter}, \cite{bourdonkleiner} and \cite{pansu}), where, for instance, the pinched vanishing result of \cite{donnellyxavier} is discussed from this point of view. These ideas seem to be somewhat out of fashion\footnote{As far as the author can tell.} (due perhaps to the success of more algebraic approaches, e.g. \cite{lueckbook}). In light of the fact that they have proportionality built in and are closely connected to the other $L^p$'s they may warrant a closer look, especially in the context of the Dodziuk-Singer question. In the negatively curved case the question of when cohomology of $\partial_{\infty}\widetilde X$ is represented by $L^2$-cochains on $\widetilde X$ is related to the conformal dimension of the boundary $\partial_{\infty}\widetilde X$. It seems interesting to study this in the nonpositively curved case and see if one can relate it to embedding obstructions for the boundary at infinity. In addition to forming a picture of the spectrum of the Laplacian on CAT(0) complexes, clarifying the connection between $L^2$-harmonic forms and thickening obstructions in these cases should shed light on whether one can expect $L^2$-Betti numbers to give thickening obstructions more generally (without curvature assumptions), or whether this is too much to expect from $L^2$-Betti numbers. 

\begin{remark}
What we have described as the analytic Dodziuk-Singer picture is usually referred to as the Singer conjecture in the literature \cite{davisonhopf,davisokun,lueckbook}, so we'll do that in the rest of the paper. 
\end{remark} 

\subsection*{Different thickenings} Finally, we wish to give two examples illustrating why the $\mathbb Q$-aspherical manifolds constructed in this paper are interesting in their own right and may have applications independent of the relation to the Dodziuk-Singer question.

\subsubsection*{Fixed point sets of homotopically trivial $\mathbb Z/p$-actions}  In the course of constructing a $\mathbb Q$-aspherical $(d+3)$-manifold $F^{d+3}$ rationally equivalent to the product of $d$ punctured tori $(\dot{\mathbb T}^2)^d$ we need to invert only a finite number of primes, so $F^{d+3}$ is actually $\mathbb Z/p$-aspherical for sufficiently large primes $p$. Thus, it is potentially the fixed point set of a homotopically trivial $\mathbb Z/p$-action\footnote{If $d-3$ is even. Otherwise, build a $(d+4)$-manifold instead.} on a (genuinely aspherical) $2d$-manifold $N^{2d}$ homotopy equivalent to $(\dot{\mathbb T}^2)^d$. One can show (see e.g. \cite{avramidil2isometries}) that the manifold $(\dot{\mathbb T}^2)^d$ does not have such homotopically trivial $\mathbb Z/p$-actions, but it is conceivable that an exotic\footnote{Such a manifold would be homotopy equivalent but not properly homotopy equivalent to $(\dot{\mathbb T}^2)^d$.} such $N^{2d}$ can be constructed via a sort of converse Smith theory (in the sense of \cite{jones}). Building the fixed point set $F$ is a first step in this direction.  

\subsubsection*{Twisted rational thickenings}
For $\Gamma=F_2^d$, different $2d$-dimensional thickenings $M$ of $B\Gamma$ can be distinguished via the fundamental class $e\in H_d(M;D)$ by taking the self-intersection $e\cap e\in H_0(M;D\otimes D)=(D\otimes D)_{\Gamma}$. This turns out to be an infinitely generated abelian group, so the question arises which of these self-intersections are represented by genuine $2d$-dimensional thickenings. (It is not hard to see that the $d$-fold cartesian product of thrice-punctured spheres is not properly homotopy equivalent to $(\dot{\mathbb T}^2)^d$, but one expects many more interesting thickenings.) Many of the difficulties of realizing elements of $(D\otimes D)_{\Gamma}$ by thickenings disappear rationally, so one can view the construction of such rational $2d$-thickenings as a first approximation of a classification of $2d$-dimensional manifolds homotopy equivalent to $BF_2^d$ up to proper homotopy equivalence. In this paper we did the untwisted case\footnote{Which, in particular, produces a $2d$-manifold rationally equivalent to $(\dot{\mathbb T}^2)^d$ but with very different peripheral properties: its interior can be homotoped to the end, in sharp contrast to the classical $(\dot{\mathbb T}^2)^d$ manifold.} $e\cap e=0$, but the methods apply more generally. 

\subsection*{Contents of the paper} 
We build a complex in section \ref{rationalhomotopysection}, show it has Poincare duality in section \ref{dualitysection}, construct a manifold-with-boundary in section \ref{surgerysection}, glue the manifolds by partially doubling along the boundary to get a closed manifold in section \ref{davissection}, and investigate its $L^2$-Betti numbers in section \ref{l2betti}. The appendix contains an alternative to section \ref{rationalhomotopysection} in a special case. 

\subsection*{Acknowledgements} I would like to thank Shmuel Weinberger for a number of discussions that motivated this paper. I would also like to thank Mladen Bestvina and Kevin Schreve for helpful disucussions. 

\section{\label{rationalhomotopysection}The equivariant Moore space problem and rational homotopy theory}
\subsection{Equivariant Moore spaces}
Let $0\ra F_{r+d}\stackrel{f_{r+d}}\ra\dots\stackrel{f_{r+1}}\ra F_r\ra D\ra 0$ be resolution of a $\Gamma$-module $D$ by finitely generated free $\mathbb Z\Gamma$-modules. The following is a version of Steenrod's equivariant Moore space question.
\begin{question}
Is there a $\Gamma$-complex $V$ with reduced chain complex $(F_*,f_*)$?
\end{question}
The main result of this section (Theorem \ref{rationalmoorespace} below) says that {\it rationally} the answer is yes for any $r\geq 2$.
After posting the first version of this paper, I was informed 
that a rational realization result of this type has already been proved by J. Smith in \cite{smith} (see Corollary 1.7 in that paper, which, more generally, proves a rational realization result for $\Gamma$-chain complexes whose homology is not necessarily concentrated in a single dimension). Our proof is somewhat different from that of \cite{smith} in that it obtains the rational realization result as a consequence of the {\it cellular Lie model} for rational homotopy (which was not available when \cite{smith} was written.) 
\begin{theorem}
\label{rationalmoorespace}
For any $r\geq 2$ there is a $\Gamma$-complex $V$ whose reduced cellular chain complex $(\overline C_*(V),\partial_V)$ is $(F_*,Nf_*)$ for some large positive integer $N$.  
\end{theorem}
\begin{remark}
So, $V$ is a simply connected $\Gamma$-complex with $\overline H_r(V;\mathbb Q)\cong D\otimes\mathbb Q$ and no other rational homology. 
\end{remark}

We wish to build a $\Gamma$-complex whose reduced cellular chain complex is $F_*$ inductively, one skeleton at a time. To do this we need to know that at each step $F_{n+1}$ can be attached to $F_{\leq n}$ along the boundary map. So
\begin{itemize}
\item
the Hurewicz map $\pi_n(F_{\leq n})\ra H_n(F_{\leq n})$ needs to be onto and, moreover,
\item
the $n$-th homotopy group $\pi_n(F_{\leq n})$ should split as a direct sum of $\Gamma$-modules $\pi_n(F_{\leq n})\cong K\oplus H_n(F_{\leq n})$ with projection onto the second factor given by the Hurewicz map.
\end{itemize}
In other words, we need to know that {\it the Hurewicz map has a $\Gamma$-equivariant section and after attaching cells via this section, the Hurewicz map for the resulting complex again has a $\Gamma$-equivariant section}. Proposition \ref{rationalsplitting} below says that we can always do this rationally. It is a consequence of a {\it homological} description of rational homotopy groups which we recall in the following subsection.

\begin{remark}
While one can arrange that the Hurewicz map is rationally onto fairly directly, showing that the resulting extension of $\mathbb Q\Gamma$-modules splits requires more work. In some special cases one can compute that the relevant group of extensions $Ext^1_{\Gamma}(H_n(F_{\leq n}),K)$ vanishes and thus obtain the desired complex $V$. We do this in the appendix. However, it turns out that the extension will rationally split even if the relevant Ext group does not vanish. 
\end{remark}

\subsection{Rational homotopy theory}
The main result of the Lie algebra approach to rational homotopy theory (Theorem I in \cite{quillen}) is that the rational homotopy groups of a simply connected space $X$ can be computed as the homology 
\begin{equation}
\label{rationalhomotopy}
H_*(L_X,\partial_X)=\pi_{*}(\Omega X)\otimes\mathbb Q=\pi_{*+1}(X)\otimes\mathbb Q
\end{equation}
of a certain differential graded Lie algebra $(L_X,\partial_X)$ that is constructed from the space $X$.
Moreover, a basepoint preserving map of simply connected spaces $f:X\ra Y$ gives a chain map of the corresponding differential graded Lie algebras $Lf:(L_X,\partial_X)\ra(L_Y,\partial_Y)$ whose homology $HLf$ is the induced map on rational homotopy groups. Consequently, if $X$ has a basepoint preserving $\Gamma$-action then (\ref{rationalhomotopy}) is a $\mathbb Q\Gamma$-module isomorphism, and we can compute the $\Gamma$-action on the rational homotopy groups from the $\Gamma$-action on the Lie algebra.

\subsection{The cellular Lie model of a cell complex}
There are several natural ways to associate to a simply connected space $X$ a differential graded Lie algebra model that computes the rational homotopy groups of that space. We use the one given in \cite{baues}, which is usually called the cellular Lie model of a cell complex $X$. We describe it in the situation when $X$ has a single vertex and no $1$-cells, which is the only case we will need. From now on $(L_X,\partial_X)$ will always denote this model. 
 
\subsection{The Lie algebra $L_X$}
Let $V$ be a graded vector space and $T(V)$ its tensor algebra. The {\it free graded Lie algebra $L(V)$} is the graded Lie subalgebra generated by $V\subset T(V)$ and the Lie bracket $[x,y]:=x\otimes y-(-1)^{|x||y|}y\otimes x$. Let $s^{-1}\overline C(X)$ be the reduced rational cellular chain complex of $X$ shifted down in dimension by one. Then $L_X$ is the free Lie algebra on $s^{-1}\overline C(X)$. A $\Gamma$-action on $X$ makes $L_X$ into graded Lie algebra of $\mathbb Q\Gamma$-modules.  

\subsection{\label{attaching}The differential $\partial_X$}
Theorem 4.11 of \cite{baues} gives an explicit chain level formula for the differential $\partial_X$ on $L_X$. We only need to know that the differential gives the cell attaching maps on the level of rational homotopy groups, i.e. if $f:(D^{n+1},S^n)\ra X$ is the attaching map of a cell $c\in C_{n+1}(X)$ then $[\partial_X (s^{-1}c)]=f_*[S^n]\in H_{n-1}(L_{X^{(n)}},\partial_{X^{(n)}})\cong \pi_n(X^{(n)})\otimes\mathbb Q$.

\begin{remark}
A direct corollary of the description of rational homotopy groups as the homology of the cellular Lie model is the following: Suppose $X$ is a simply connected space whose reduced cellular chain complex has no cells in dimension $<k$. Then the rational Hurewicz map $\pi_{*}(X)\otimes\mathbb Q\ra H_*(X;\mathbb Q)$ is an isomorphism in dimensions $\leq 2k-2$.\footnote{This follows because rational homology and rational homotopy get identified by the Hurewicz map since there are no Lie brackets in dimensions $<2k-1$.} This is a special case of the rational Hurewicz theorem (first proved by Serre), which we will need in section \ref{surgerysection}. 

\begin{theorem}[Rational Hurewicz theorem \cite{klauskreck}]
Suppose that $X$ is simply connected and that $\overline H_{<k}(X;\mathbb Q)=0$. Then the rational Hurewicz map 
\begin{equation}
\label{qhurewicz}
\pi_*(X)\otimes\mathbb Q\ra H_*(X;\mathbb Q)
\end{equation} 
is an isomorphism in dimensions $\leq 2k-2$. 
\end{theorem}
\end{remark}
\subsection{Hurewicz map on chain complexes}
The rational Hurewicz map on the level of chain complexes is the chain map $h_X:L_{X}\ra C(X)$ obtained by setting all non-trivial Lie brackets to zero. Passing to homology, we get the map (\ref{qhurewicz}) above. A {\it section} of $h_X$ is a chain map $s:C(X)\ra L_{X}$ with $h_X\circ s=id_{C(X)}$. 

The homological description of rational homotopy groups can be used to turn the free resolution $F_*\ra D$ into a $\Gamma$-chain complex. The main step is expressed in the following proposition. 
\begin{proposition}
\label{rationalsplitting}
Suppose $X^n$ is an $n$-dimensional $\Gamma$-complex with one vertex, no $1$-cells, and 
\begin{itemize}
\item
the Hurewicz map $h_{X^n}:L_{X^n}\ra C(X^n)$ has a $\Gamma$-equivariant section $s$.
\end{itemize}
We can $\Gamma$-equivariantly attach $(n+1)$-cells to $X^n$ to get a $\Gamma$-complex $X^{n+1}=X^n\cup F_{n+1}$ so 
\begin{itemize}
\item
$H_n(X^{n+1};\mathbb Q)=0$, and
\item
the Hurewicz map $h_{X^{n+1}}:L_{X^{n+1}}\ra C(X^{n+1})$ has a $\Gamma$-equivariant section. 
\end{itemize}
\end{proposition}
\subsection*{Notation}Everywhere in the proof below $H_*$ and $\pi_*$ will denote the rational homology and rational homotopy groups, respectively. We will omit $\mathbb Q$ from the notation from now on.
\begin{proof}
Let $(F_{\leq n},\partial_F)$ be the rational cellular chain complex $C(X^n)$. Also, write the chain complex $L_{X^n}$ as $(L_*\oplus F_*,\partial)$ where $L_*$ is the subspace of the cellular Lie model that involves at least one Lie bracket. The section $s$ restricts to a chain $\Gamma$-map $s_n:F_n\ra L_n\oplus F_n$.
Let $Z_n=\ker(\partial_F)=H_n(X^n)$ be the subspace of cycles in $F_n$. Since $s_n$ is a chain map, for any cycle $z\in Z_n$ we have $\partial s_n(z)=s_n(\partial_F(z))=0$, so $s_n$ further restricts to a map 
$$
Z_n\stackrel{s_n}\longrightarrow\ker\partial.
$$
Denote by $\overline s_n$ its composition with the quotient map 
$$
\ker\partial\stackrel{q}\longrightarrow{\ker\partial\over\partial L_{n+1}}=\pi_n(X^n).
$$
Next, pick $\Gamma$-generators $F_{n+1}\stackrel{\partial_F}\twoheadrightarrow Z_n$. We can take multiples of these generators, if necessary, so that $\overline{s_n}\circ\partial _F$ consists of integral elements in $\pi_n(X^n)$. Now, we can $\Gamma$-equivariantly attach $(n+1)$-cells to $X^n$ to get a $\Gamma$-complex $X^{n+1}=X^n\cup F_{n+1}$ whose attaching maps on rational homotopy are given by $\overline{s_n}\circ\partial_F:F_{n+1}\ra\pi_{n}(X^n)$. Look at the resulting complex $X^{n+1}$. 
\begin{itemize}
\item
Note that we killed all the $n$-dimensional rational homology of $X^n$, i.e. $H_{n}(X^{n+1})=0$.
\end{itemize}
Also note that the $(n+1)$-cells do not contribute any Lie brackets in dimensions $\leq n+1$, so in dimension $\leq n+1$ the only difference between the chain complexes $L_{X^n}$ and $L_{X^{n+1}}$ are the extra cells $F_{n+1}$. Let $\partial:F_{n+1}\longrightarrow\ker\partial$ be the boundary map (of the cellular Lie model $L_{X^{n+1}}$) restricted to these cells, and $\overline\partial=q\circ\partial$ its image in $\pi_n(X^n)$. 

The two maps 
$\partial$ and $s_n\circ\partial_F$
might not be the same. However, the description of both $\overline\partial$ and the composition $\overline{s_n}\circ\partial_F$ as the attaching map on the level of rational homotopy groups implies 
\begin{equation}
\overline s_n\circ\partial_F=\overline\partial:F_{n+1}\ra\pi_n(X^n)={\ker\partial\over\partial L_{n+1}}.
\end{equation}
Consequently, for any $x\in F_{n+1}$ there is $y\in L_{n+1}$ so that $s_n\partial_F(x)=\partial(x)+\partial(y)$. Since $F_{n+1}$ is a free $\Gamma$-module, we can pick such a $y$ for each generator $x$ of $F_{n+1}$ to get a $\Gamma$-map $\phi_{n+1}:F_{n+1}\ra L_{n+1}$ satisfying 
\begin{equation}
s_n\partial_F(x)=\partial(x)+\partial\phi_{n+1}(x).
\end{equation} 
Finally, we let 
\begin{eqnarray}
F_{n+1}&\stackrel{s_{n+1}}\longrightarrow& L_{n+1}\oplus F_{n+1},\\
x&\mapsto&\phi_{n+1}(x)+x.
\end{eqnarray}
Then $\partial(s_{n+1}(x))=\partial\phi_{n+1}(x)+\partial(x)=s_n(\partial_F(x))$ shows that the map $s:C(X^{n+1})\ra L_{X^{n+1}}$ is a chain map. Moreover, it is clearly a $\Gamma$-map and a section of $h_{X^{n+1}}$. This finishes the proof of the proposition.
\end{proof}

\begin{proof}[Proof of Theorem \ref{rationalmoorespace}]
Now let $V^r$ be a wedge of $r$-spheres with a $\Gamma$-action representing the free $\Gamma$-module $F_r$. Of course, $\pi_r(V^r)=H_r(V^r)=F_r$ and we can keep applying the Proposition to construct a $\Gamma$-complex $V:=V^{r+d}$ whose reduced chain complex is $F_*$ and whose boundary maps are multiples\footnote{Since each $F_{n+1}$ is a finitely generated free $\mathbb Z\Gamma$-module we can multiply all the $\mathbb Q\Gamma$-generators of $H_n(X^n)$ by a single positive integer $N$ so that all of them are represented by integral elements of $\pi_n(X^n)$.} of the boundary maps in the complex $(F_*,f_*)$. 
\end{proof}


\begin{remark}
A good reference for rational homotopy theory is \cite{felixhalperinthomas}. However, it takes Sullivan's cohomological approach \cite{sullivan} as its starting point and thus deals mostly with spaces of finite type. Instead, we need Quillen's Lie algebra approach to rational homotopy theory which works for general simply connected spaces. The existence of a Lie model for such spaces is Theorem I of \cite{quillen}. However, the particular Lie model we use (called the {\it cellular Lie model} of a cell complex) came later and can be found, for example, in \cite{baues}, where it is also shown (see 4.9(4) of \cite{baues}) that it is equivalent to Quillen's model. The cellular Lie model in the finite type case can be found in 24(e) of \cite{felixhalperinthomas}.
\end{remark}


\subsection{Locally finite complexes}
We replace $V$ by a locally finite complex $\widetilde X$ with free $\Gamma$-action. 
\begin{corollary}
\label{lfcomplex}
There is a finite complex $X$ with fundamental $\Gamma$ and universal cover $\widetilde X$ whose cellular chain complex is $(C_*(\widetilde X),\partial_{\widetilde X})=(C_*(E\Gamma)\oplus\overline C_*(V),\partial_{E\Gamma}\oplus\partial_V)$. 
\end{corollary}
\begin{proof}
Pick finitely many generators $e_1,\dots,e_n$ for $F_r$, attach them as a finite wedge of spheres $\vee_{i=1}^nS^r$ at a vertex of $E\Gamma$ and denote by $E\Gamma\cup F_r:=E\Gamma\cup\Gamma(\vee_{i=1}^n S^r)$ the $\Gamma$-orbit of this. Next, equivariantly attach the cells $F_{r+1}$ (we can do this because $F_r=H_r(E\Gamma\cup F_r)=\pi_r(E\Gamma\cup F_r)$) and let $E\Gamma\cup F_r\cup F_{r+1}\ra V$ be the $\Gamma$-map obtained by collapsing $E\Gamma$ to a point. On rational homotopy groups we have an exact sequence of $\Gamma$-modules 
$$
F_{r+2}\ra\pi_{r+1}(E\Gamma\cup F_r\cup F_{r+1})\ra\pi_{r+1}(V)\ra 0,
$$
so we can attach $F_{r+2}$ to kill it off, extend to a $\Gamma$-map $E\Gamma\cup F_r\cup F_{r+1}\cup F_{r+2}\ra V$ an continue in this manner to get a locally finite complex $\widetilde X:=E\Gamma\cup F_r\cup\dots\cup F_{r+d}$ with a free $\Gamma$-action and a rational equivalence $\widetilde X\ra V$. Let $X=\widetilde X/\Gamma$ be the quotient. 
\end{proof}
\section{\label{dualitysection}Duality groups and Poincare spaces}
\subsection{Duality groups} 
Suppose that $\Gamma$ has a finite classifying space $B\Gamma$. It is a {\it $d$-dimensional duality group} if $H^*(B\Gamma;\mathbb Z\Gamma)$ is non-zero only in dimension $d$. The module $D\cong H^d(B\Gamma;\mathbb Z\Gamma)$ is called the {\it dualizing module} of the group $\Gamma$. Note that the the $\mathbb Z\Gamma$-modules $C^*(B\Gamma;\mathbb Z\Gamma)=Hom_{\Gamma}(C_*(E\Gamma),\mathbb Z\Gamma)=C_*(E\Gamma)^*$ used in computing group cohomology with group ring coefficients are free and finitely generated (because $B\Gamma$ is finite) so the dualizing module has a natural resolution by free, finitely generated $\mathbb Z\Gamma$-modules 
$$
\begin{array}{ccccccccccc}
0&\ra& C^{0}(B\Gamma;\mathbb Z\Gamma)&\ra&\dots&\ra &C^d(B\Gamma;\mathbb Z\Gamma)&\ra& H^d(B\Gamma;\mathbb Z\Gamma)&\ra& 0,\\
&&||&&&&||&&||&&\\
0&\ra&F_{r+d}&\ra&\dots&\ra&F_r&\ra&D&\ra&0.
\end{array}
$$

\begin{remark}
If $\Gamma$ is a finite type $d$-dimensional duality group and $D$ is its dualizing module, then for the complex $\widetilde X$ constructed in Corollary \ref{lfcomplex} the group of cellular chains has the form
$$
C_*(\widetilde X)=C_*(E\Gamma)\oplus C_{r+d-*}(E\Gamma)^*.
$$ 
This suggests $X$ is some kind of self-dual object of dimension $r+d$. The main result of this section (Theorem \ref{poincarepair}) says that this is, in fact, the case. 
\end{remark}
\begin{example}
Finite index torsion free subgroups of lattices in locally symmetric spaces, mapping class groups and $\Out(F_n)$ are prominent examples of duality groups.
\end{example}
\subsection{Poincare complexes and Poincare pairs}
A finite complex $X$ is an {\it $n$-dimensional Poincare complex} if $H_n(X;\mathbb Z)\cong\mathbb Z$ and the generator $[X]\in H_n(X;\mathbb Z)$ defines a cap product isomorphism
\begin{equation}
\label{poincare}
H^*(X;\mathbb Z\Gamma)\stackrel{\cap [X]}\cong H_{n-*}(X;\mathbb Z\Gamma).
\end{equation}
A $\pi_1$-isomorphism $f:X\ra Y$ of finite, connected complexes is an {\it $(n+1)$-dimensional Poincare pair\footnote{We can always replace the map $f:X\ra Y$ by an inclusion by taking the mapping cylinder.}} if $H_{n+1}(Y,X;\mathbb Z)\cong\mathbb Z$ and the generator $[Y,X]\in H_{n+1}(Y,X;\mathbb Z)$ defines cap product isomorphisms
\begin{eqnarray}
\label{lefshetz1}
H^*(Y;\mathbb Z\Gamma)&\stackrel{\cap[Y,X]}\cong& H_{n+1-*}(Y,X;\mathbb Z\Gamma),\\
\label{lefshetz2}
H^*(Y,X;\mathbb Z\Gamma)&\stackrel{\cap[Y,X]}\cong&H_{n+1-*}(Y;\mathbb Z\Gamma).
\end{eqnarray}
Rational Poincare complexes and rational Poincare pairs are defined in exactly the same way with $\mathbb Z$ replaced by $\mathbb Q$ everywhere.
\begin{remark}
It follows from this that the connecting homomorphism $H_{n+1}(Y,X;\mathbb Z)\stackrel{\partial}\ra H_n(X;\mathbb Z)$ is an isomorphism and if we let $[X]=\partial[Y,X]$ then $X$ is an $n$-dimensional Poincare complex with Poincare duality given by $\cap[X]$. 
\end{remark} 


\subsection{Stabilization}
If $(M,\partial)$ is a compact aspherical $n$-manifold-with-boundary then $(M\times I, \partial(M\times I))=(M\times I,M\cup_{\partial} M)$ is a compact aspherical $(n+1)$-manifold with boundary. Suppose $\partial$ is connected and $\partial\hookrightarrow M$ is a $\pi_1$-isomorphism. Then the connecting map in the Mayer-Vietoris sequence $\overline H_{*+1}(\widetilde M\cup_{\widetilde \partial}\widetilde{M})\ra\overline H_*(\widetilde\partial)$ is an isomorphism of $\Gamma$-modules. This motivates the following ``stabilization'' procedure. 

\begin{definition}
Suppose that $X$ is a finite, connected complex with fundamental group $\Gamma$ and 
classifying map $X\ra B\Gamma$. Let $SX:=B\Gamma\cup_XB\Gamma$ with classifying map $SX\ra B\Gamma$ and call this the suspension of $X$. Denote by $S^kX\ra B\Gamma$ the complex obtained from $X$ by applying this suspension $k$-times. 
\end{definition} 
\begin{itemize}
\item
The Mayer-Vietoris sequence in homology of the universal cover gives an isomorphism of $\Gamma$-modules 
\begin{equation}
\label{stablemodule}
\overline H_{*+1}(\widetilde{SX};\mathbb Z)\cong\overline H_{*}(\widetilde{X};\mathbb Z).
\end{equation}
\item
The relative Mayer-Vietoris sequence in cohomology with group ring coefficients gives the isomorphism 
\begin{equation}
\label{relativecohomology}
H^*(B\Gamma,X;\mathbb Z\Gamma)\cong H^{*+1}(B\Gamma,SX;\mathbb Z\Gamma).
\end{equation}
\end{itemize}
Both of these bullets are also valid with $\mathbb Z$ replaced by $\mathbb Q$. 
\begin{lemma}[Poincare stabilization]
If $SX\ra B\Gamma$ is an $(n+1)$-dimensional (rational) Poincare pair then $X\ra B\Gamma$ is an $n$-dimensional (rational) Poincare pair.
\end{lemma}
\begin{proof}
First, since $(B\Gamma,SX)$ is an ($n+1$)-dimensional Poincare pair the relative cohomology $H^*(B\Gamma,SX;\mathbb Z\Gamma)\cong H_{n+1-*}(B\Gamma;\mathbb Z\Gamma)$ is concentrated in a single dimension $n+1$ and is equal to $\mathbb Z $. By the isomorphism (\ref{relativecohomology}), $H^*(B\Gamma,X;\mathbb Z\Gamma)$ is concentrated in dimension $n$ and equals $\mathbb Z$. We will first check that the cap product gives an isomorphism between this group and $H_{n-*}(B\Gamma;\mathbb Z\Gamma).$

Since $(B\Gamma,SX)$ is an $(n+1)$-dimensional Poincare pair, the suspension $SX$ is an $n$-dimensional Poincare complex. We will look at the long exact sequences in homology and cohomology (with group ring coefficients) for the pair $(SX,B\Gamma)$ where $B\Gamma\hookrightarrow SX$ is the inclusion of the left copy into $SX=B\Gamma\cup_XB\Gamma$. 

In low dimensions the long exact sequences fit into the diagram
\begin{equation}
\begin{array}{ccccc}
\mathbb Z&&&&\\
||&&&&\\
H^n(B\Gamma,X;\mathbb Z\Gamma)&&&&0\\
||&&&&||\\
H^n(SX,B\Gamma;\mathbb Z\Gamma)&\ra& H^n(SX;\mathbb Z\Gamma)&\ra&H^n(B\Gamma;\mathbb Z\Gamma),\\
\downarrow&&||&&\\
H_0(B\Gamma;\mathbb Z\Gamma)&=&H_0(SX;\mathbb Z
\Gamma)&&\\
||&&||&&\\
\mathbb Z&&\mathbb Z&&
\end{array}
\end{equation}
with the middle vertical isomorphism the Poincare duality isomorphism for $SX$ and the top left vertical isomorphism coming from excision. Moreover, $H^n(B\Gamma;\mathbb Z\Gamma)\cong H_1(B\Gamma,SX;\mathbb Z\Gamma)=0$ because $SX\ra B\Gamma$ is a $\pi_1$-isomorphism of connected complexes. This diagram shows that cap product map $H^n(B\Gamma,X;\mathbb Z\Gamma)\ra H_0(B\Gamma;\mathbb Z\Gamma)$ is an isomorphism. 

It remains to check that the cap product maps $H^{n-*}(B\Gamma;\mathbb Z\Gamma)\ra H_{*}(B\Gamma,X;\mathbb Z\Gamma)$ are also isomorphisms. For $*=0$ this follows because $H_0(B\Gamma,X;\mathbb Z\Gamma)=0$ (since $X\ra B\Gamma$ is a $\pi_1$-isomorphism of connected complexes).
 
To get the isomorphism for $*>0$ we again look at the long exact sequences, this time in higher dimensions. \underline{In the diagram below, the coefficients are always in the group ring $\mathbb Z\Gamma$.} They are omitted from the notation in order to save space and make the diagram (slightly) more readable.  
\begin{equation}
\label{relative}
\begin{array}{ccccccc}
&&&&&&H^{n+1-*}(B\Gamma,X)\\
&&&&&&||\\
&\ra& H^{n-*}(SX)&\ra&H^{n-*}(B\Gamma)&\ra& H^{n+1-*}(SX,B\Gamma)\\
&&||&&\downarrow&&\downarrow\\
H_{*}(B\Gamma)&\ra&H_{*}(SX)&\ra&H_{*}(SX,B\Gamma)&\ra&H_{*-1}(B\Gamma),\\
||&&&&||&&\\
0&&&&H_{*}(B\Gamma,X)&&
\end{array}
\end{equation}
For $*=1$, the right column of the above diagram consists of $\mathbb Z$'s and the right horizontal maps are the zero maps, so we get a cap product isomorphism $H^{n-1}(B\Gamma;\mathbb Z\Gamma)\cong H_1(B\Gamma,X;\mathbb Z\Gamma)$. Finally, for $*>1$ everything in the right column is zero, so we again find that the cap product gives isomorphisms $H^{n-*}(B\Gamma;\mathbb Z\Gamma)\cong H_*(B\Gamma,X;\mathbb Z\Gamma)$. This completes the proof that $(B\Gamma,X)$ is an $n$-dimensional Poincare pair. 

The same argument clearly work rationally with $\mathbb Z$ replaced by $\mathbb Q$ everywhere. 
\end{proof}
\begin{remark}
Looking at the same diagrams will convince the reader of the converse: If $X\ra B\Gamma$ is an $n$-dimensional Poincare pair, then $SX\ra B\Gamma$ is an $(n+1)$-dimensional Poincare pair.\footnote{We don't actually need this direction for our argument.}
\end{remark}

\begin{theorem}
\label{poincarepair}
Let $\Gamma$ be a $d$-dimensional duality group with finite classifying space $B\Gamma$ and dualizing module $D$. Let $X$ be a finite complex with fundamental group $\Gamma$. Suppose that $\overline H_*(\widetilde X;\mathbb Q)$ is concentrated in dimension $r\geq 2$ and $H_r(\widetilde X;\mathbb Q)\cong D\otimes\mathbb Q$ as a $\Gamma$-module.\footnote{Such an $X$ always exists by Corollary \ref{lfcomplex}.} Then,
\begin{itemize}
\item
for $r>d$ all such $X$ with $\overline H_r(\widetilde X;\mathbb Q)\cong D\otimes\mathbb Q$ are rationally equivalent, and
\item
any such $X\ra B\Gamma$ is a $(d+r+1)$-dimensional rational Poincare pair. 
\end{itemize}
\end{theorem}
\begin{proof}
To prove the first bullet, we will construct a rational equivalence from a standard rational model $X_0$ whose universal cover has cell complex $C_*(E\Gamma)\oplus F_*$ with boundary maps that are positive integer multiples of $\partial_{E\Gamma}\oplus f_*$. We will do this one skeleton at a time, via obstruction theory. The key point is that $\widetilde X$ is rationally homotopy equivalent to a wedge of $r$-spheres $\vee S^r$ so $\pi_{r+k}(X)$ is a torsion group of bounded exponent for all $0<k<r-1$.
 
We need to define the map $X_0\ra X$. First, we can build a $\pi_1$-isomorphism $X_0^{(r-1)}\ra X$ because $\widetilde X$ is rationally $(r-1)$-connected (if necessary, we replace the attaching maps in $C_*(X_0)$ by multiples in order to do this). Next, we extend it to $\widetilde\phi_r:F_r\ra Z_r(\widetilde X)$ via the $\Gamma$-map that maps in the generators of $D=H_r(\widetilde X)$. The obstruction to extending the map $\phi_r:X_0^{(r)}\ra X$ to the $(r+1)$-skeleton lies in the cohomology group $H^{r+1}(X_0;\pi_r(X))$. Since (some integer multiple of) the image of $F_{r+1}\ra F_r$ is in the kernel of the composite map $F_r\ra Z_r(\widetilde X)\ra H_r(\widetilde X)=\pi_r(X)=D$ the obstruction vanishes and we can extend to a $\Gamma$-map $\widetilde \phi_{r+1}:F_{r+1}\ra C_{r+1}(\widetilde X)$. Further obstructions to extending the map $\phi_{r+k}:X_0^{(r+k)}\ra X$ from the $(r+k)$-skeleton to the $(r+k+1)$-skeleton lie in $H^{r+k+1}(X_0;\pi_{r+k}(X))$ and for $0<k<r-1$ all of the coefficients in these cohomology groups are torsion groups of bounded exponent. Thus, after taking multiples of the attaching maps in $X_0$ if necessary, we can make sure that all these obstructions vanish and extend to a map $X_0^{(2r-1)}\ra X$. Since $d<r$, this constructs a $\pi_1$-isomorphism from an $(r+d)$-dimensional complex $X_0\ra X$. Since we constructed it to be an isomorphism on rational cohomology of universal covers $H_*(\widetilde X_0;\mathbb Q)\cong H_*(\widetilde X;\mathbb Q)$, the map is a rational equivalence. 

For any two such complexes $X$ and $Y$, we may take the standard models to be equal ($X_0=Y_0$) because we can take common multiples of the boundary maps. Thus the complexes $X$ and $Y$ are rationally equivalent, which proves the first bullet.

To prove the second bullet, embedd $B\Gamma$ in a high dimensional Euclidean space $\mathbb R^{d+r'+1}$ with $r'>d$ and $r'\geq r$. Then the regular neighborhood of $B\Gamma$ is a compact aspherical $(d+r'+1)$-manifold-with-boundary $(M,\partial)$ with fundamental groups $\pi_1M=\pi_1\partial=\Gamma$. The long exact homology sequence for $(\widetilde M,\widetilde\partial)$, Lefshetz duality, and the interpretation of compactly supported cohomology of the universal cover of a finite complex as the cohomology with group ring coefficients 
\begin{equation}
\label{duality}
\overline H_*(\widetilde\partial)\cong H_{*+1}(\widetilde M,\widetilde\partial)\cong H^{d+r'-*}_c(\widetilde M)\cong H^{d+r'-*}(M;\mathbb Z\Gamma)\cong H^{d+r'-*}(B\Gamma;\mathbb Z\Gamma)
\end{equation}
implies that $\overline H_*(\widetilde\partial)$ is concentrated in dimension $r'$ and equals the dualizing module $D$. On the other hand, (\ref{stablemodule}) implies that $\overline H_{*}(\widetilde{ S^{r'-r}X};\mathbb Q)$ is also concentrated in dimension $r'$ and equals the rational dualizing module $D\otimes\mathbb Q$. Thus, the first bullet implies $\partial\ra B\Gamma$ is rationally equivalent to the stabilization $S^{r'-r}X\ra B\Gamma$. Since $\partial\ra B\Gamma$ is a rational $(d+r'+1)$-dimensional Poincare pair, $S^{r'-r}X\ra B\Gamma$ is as well. So, by the above Lemma $X\ra B\Gamma$ is a rational $(d+r+1)$-dimensional Poincare pair.
\end{proof}

\begin{remark}
Informally speaking, Theorem \ref{poincarepair} says that all rational models for the boundary of $B\Gamma$ are stably rationally equivalent. Unstably this is not true. For instance, for the model $X$ constructed in Corollary \ref{lfcomplex}, the classifying map $X\ra B\Gamma$ has a section (it is an ``untwisted model''). On the other hand, the classical $2d$-dimensional manifold models $(M,\partial)$ for $BF_2^d$ that occur in nature (Cartesian products of surfaces-with-boundary) have the property that the interior $M$ cannot be homotoped into the boundary $\partial$, and so the classifying map $\partial\ra BF_2^d$ does not have a homotopy section. The first obstruction is an element of $H^d(M;D)\cong H_0(M;D\otimes D)$ which can be geometrically interpreted as a ``self-intersection'' $e\cap e$ of the fundamental class $e\in H_d(M;D)$.
\end{remark}


\section{\label{surgerysection}Surgery}
In the first two subsections we sketch how surgery theory works. Good introductions are \cite{luecksurgery, ranicki} and the first half of \cite{weinbergerbook}, while the classic reference is \cite{wall}. 
\subsection{Constructing closed manifolds via surgery}
The surgery method for constructing closed manifolds homotopy equivalent to a given finite complex $X$ consists of three steps.
\begin{enumerate} 
\item 
Verify that $X$ has the global Poincare duality of an $n$-dimensional manifold (i.e. that $X$ is a Poincare complex). 

\item Find a candidate for the normal bundle. That is, find a disk bundle $D^k\ra\nu\ra X$ and a degree one\footnote{Since $X$ is a Poincare complex the Thom space of $\nu$ has $H_{n+k}(Th(\nu))\cong\mathbb Z$. A map $f:S^{n+k}\ra Th(\nu)$ has degree one if $f_*[S^{n+k}]$ generates $H_{n+k}(Th(\nu))$.} map $f:S^{n+k}\ra Th(\nu)$ to the Thom space of this bundle.\footnote{For large enough $k$, we can embedd $X\hookrightarrow S^{n+k}$. If $X$ is an actual manifold, then its regular neighborhood is an actual bundle $\nu$ and collapsing the complement of the regular neighborhood to a point gives a degree one map $S^{n+k}\ra Th(\nu)$.} This bundle is pulled back from the universal $k$-plane bundle over a Grassmannian of $k$-planes $\nu_k\ra Gr_k$. The transverse inverse image of the zero section $Gr_k\hookrightarrow Th(\nu_k)$ under the map $S^{n+k}\ra Th(\nu)\ra Th(\nu_k)$ gives a manifold $M$ and a degree one bundle map $M\ra X$. 
\item 
Modify $M$ and this map to make it into a $\pi_1$-isomorphism and inductively kill off the homology kernel (doing surgery modifications to the domain $M$ while keeping it a manifold) to finally get a manifold homotopy equivalent to $X$.
\end{enumerate}
 
All three of these steps might fail. The complex $X$ might not have Poincare duality, there might not be any candidate normal bundle\footnote{That is, for every bundle the Thom class is not spherical.} and one might not be able to do the surgery modifications to kill off the homology kernel.\footnote{The obstructions to completing this last step are the well-studied surgery obstruction groups (Wall's $L$-groups). We will not need them in this paper.}

\subsection{\label{construction}Constructing compact manifolds-with-boundary}
One can also look at the relative situation of constructing compact $(n+1)$-dimensional manifolds-with-boundary homotopy equivalent to a given pair of finite complexes $f:X\ra Y$. First, we need to check that $(Y,X)$ is a $(n+1)$-dimensional Poincare pair. Second, one needs a candidate normal bundle $D^k\ra\nu\ra Y$, i.e. a bundle with a degree one map\footnote{This means the generator of the homology $H_{n+1+k}(Th(\nu),Th(f^*\nu))\cong\mathbb Z$ is represented by this map.} $\phi:(D^{n+k+1},S^{n+k})\ra(Th(\nu),Th(f^*\nu))$. Given such a map one takes the transverse inverse of $Gr_k$ under $(D^{n+1+k},S^{n+k})\ra Th(\nu_k)$ and gets an $(n+1)$-manifold-with-boundary $(M,\partial)$ and  a degree one bundle map $(M,\partial M)\ra (Y,X)$. 
As before one turns all the maps into $\pi_1$-isomorphisms and then tries to kill off the relevant homology kernel\footnote{That is the kernel $K_*(M,\partial)$ of the map $H_*(M,\partial;\mathbb Z\pi_1)\ra H_*(Y,X;\mathbb Z\pi_1)$.} by surgery modifications to obtain a compact manifold with boundary homotopy equivalent to $(Y,X)$. A (the?) fundamental result of non-simply connected surgery is that, for a $\pi_1$-isomorphism $f:X\ra Y$ and $n\geq 5$, this last surgery modification step always works.

\begin{theorem}[Wall's $\pi-\pi$ theorem, 3.3 in \cite{wall}]
Suppose $f:X\ra Y$ is a $\pi_1$-isomorphism, $(Y,X)$ is an $(n+1)$-dimensional Poincare pair, $n\geq 5$, and there is a disk bundle $D^k\ra\nu\ra Y$ with a degree one map $(D^{n+k+1},S^{n+k})\ra(Th(\nu),Th(f^*\nu))$. Then $(Y,X)$ is homotopy equivalent to an $(n+1)$-dimensional compact manifold with boundary $(N,\partial)$. 
\end{theorem}

\begin{remark}
A few words about the argument in the (easier) even dimensional $n+1=2r$ case: One does surgery below the middle dimension to kill homology kernels in dimensions $<r$,\footnote{This doesn't use either the fact that $(Y,X)$ is a Poincare pair, or that the map $\phi$ has degree one.} shows that, after attaching finitely many trivial $r$-handles if necessary,\footnote{This means taking the boundary-connect-sum of $(M,\partial)$ with finitely many copies of $S^r\times D^r$.} the homology kernel in dimension $r$ has a $\mathbb Z\Gamma$-basis represented by finitely many disjoint, embedded $r$-disks $(D^r,S^{r-1})\hookrightarrow (M,\partial)$,\footnote{Once $(M,\partial)\ra (Y,X)$ is a sufficiently connected degree one map of Poincare pairs one shows that the kernel $K_r(M,\partial)$ is a finitely generated, stably free $\mathbb Z\pi_1$-module. Attaching enough trivial $r$-handles makes it actually free, hence represented by immersed disks with isolated intersections and self-intersections. The $\pi_1$-isomorphism $X\ra Y$ lets one remove these intersections.} and cuts out small neighborhoods of these $r$-disks to get a manifold-with-boundary homotopy equivalent to $(Y,X)$. This method is sometimes called {\it handle subtraction}.
\end{remark}

\subsection{Rational surgery}
The previous remark suggests that Wall's argument works just as well for rational Poincare pairs and finite degree maps. This is, in fact, the case and goes by the name {\it surgery with coefficients} \cite{anderson} or {\it local\footnote{The word ``local'' means ``after inverting some primes''. We are doing things as locally as possible by inverting all primes.} surgery} \cite{taylorwilliams}. In particular, one gets a rational $\pi-\pi$ theorem as a direct consequence of Theorem 5.1 in \cite{anderson}.

\begin{theorem}[Rational $\pi-\pi$ theorem]
Suppose $f:X\ra Y$ is a $\pi_1$-isomorphism, $(Y,X)$ is an $(n+1)$-dimensional rational Poincare pair, $n\geq 5$, and there is a disk bundle $D^k\ra\nu\ra Y$ with a finite degree map $(D^{n+k+1},S^{n+k})\ra(Th(\nu),Th(f^*\nu))$. Then there is a $\pi_1$-isomorphism from a compact $(n+1)$-manfiold-with-boundary $(N,\partial)\ra(Y,X)$ that lifts to a rational homology isomorphism $(\widetilde N,\widetilde\partial)\ra (\widetilde Y,\widetilde X)$ of universal covers.
\end{theorem}
\begin{proof}
If $(D^{n+k+1},S^{n+k})\ra(Th(\nu),Th(f^*\nu))$ is a finite degree map to a Thom space over a rational Poincare pair, then taking the transverse inverse image of $Gr_k$ (as described in subsection \ref{construction}) gives a finite degree normal map $(M,\partial)\ra(Y,X)$. Replacing the fundamental class $[Y,X]$ by a rational multiple we get a degree one normal map to $(Y,X)$ which is, with this new fundamental class, still a rational Poincare pair. Now, Theorem 1 of chapter 5 of \cite{anderson} applies and gives the desired rational equivalence $(N,\partial)\ra(Y,X)$. 
\end{proof}

This is useful because rationally any high dimensional bundle works as the normal bundle. 

\begin{lemma}
For $k>n+1$, any $k$-disk bundle $D^k\ra\nu\ra Y$ has a finite degree map $(D^{n+k+1},S^{n+k})\ra(Th(\nu),Th(f^*\nu))$. 
\end{lemma}
\begin{proof}
The Thom spaces $Th(\nu)$ and $Th(f^*\nu)$ are $(k-1)$-connected so in the diagram 
$$
\begin{array}{ccccc}
&&\pi_{n+k}(Th(f^*\nu))&\ra&\pi_{n+k}(Th(\nu))\\
&&\downarrow&&\downarrow\\
H_{n+k+1}(Th(\nu),Th(f^*\nu))&\cong&H_{n+k}(Th(f^*\nu))&\stackrel{0}\ra&H_{n+k}(Th(\nu)).
\end{array}
$$
the vertical (Hurewicz) maps are rationally isomorphisms for $n+k\leq 2k-2$, i.e. for $k>n+1$ (by the rational Hurewicz theorem). Thus, some positive multiple of the fundamental class in $H_{n+k}(Th(f^*\nu);\mathbb Z)$ is represented by a map $\phi:S^{n+k}\ra Th(f^*\nu)$. The image of $\phi_*[S^{n+k}]$ in $Th(\nu)$ is homologous to zero, so $S^{n+k}\ra Th(f^*\nu)\ra Th(\nu)$ defines a finite order element in $\pi_{n+k}(Th(\nu))$. Thus, some positive multiple $N\phi:S^{n+k}\ra Th(f^*\nu)$ bounds a disk in $Th(\nu)$, i.e. extends to a map $(D^{n+k+1},S^{n+k})\ra(Th(\nu),Th(f^*\nu))$. This is a finite degree map (it represents a positive multiple of the generator in $H_{n+k}(Th(\nu),Th(f^*\nu))$ because its boundary represents a multiple of the generator $N\phi_*[S^{n+k}]$ in $H_{n+k}(Th(f^*\nu))$). 
\end{proof}

Putting these together we get the following result. 
\begin{theorem}[Rational surgery theorem]
\label{rationalsurgery}
If a $\pi_1$-isomorphism $f:X\ra Y$ is a rational $(n+1)$-dimensional Poincare pair, $n\geq 5$ then there is a $\pi_1$-isomorphism from a compact manifold-with-boundary $g:(M,\partial)\ra(Y,X)$ so that $\widetilde g:(\widetilde M,\widetilde\partial)\ra(\widetilde Y,\widetilde X)$ is a rational homology isomorphism. 
\end{theorem}
 
\begin{remark}
See Lemma 2 in section 3 of \cite{weinberger} for a similar argument showing that a simply connected rational Poincare complex with vanishing Witt signature\footnote{This is needed to show the rational surgery obstruction vanishes in the simply connected absolute case. It is replaced by the rational $\pi$-$\pi$ theorem in the relative $\pi_1$-isomorphism case.} is rationally equivalent to a simply connected manifold. 
\end{remark}

\subsection{Rational manifold models for classifying spaces}
Combining Corollary \ref{lfcomplex}, Theorem \ref{poincarepair} and the rational surgery theorem gives the main result of this paper. 
\begin{theorem}
\label{rationalmodels}
Let $\Gamma$ be a $d$-dimensional duality group with finite classifying space $B\Gamma$. For any $r\geq 2$ with $d+r\geq 5$ there is a compact $(d+r+1)$-manifold-with-boundary $(M,\partial)$ with $\pi_1M=\pi_1\partial=\Gamma$ and rationally acyclic universal cover $\widetilde M$. 
\end{theorem}
By contrast, there are often obstructions to constructing genuine, integral manifold classifying spaces of dimension $d+r+1$. For instance, \cite{bestvinakapovichkleiner} shows that any aspherical manifold with fundamental group $F_2^d$ has dimension $\geq 2d$.
\begin{remark}
Rationality enters at three stages. First, in solving the equivariant Moore space problem (rationally), second in showing $X\ra B\Gamma$ is a rational Poincare pair and third in finding a candidate normal bundle to perform the surgery. The result of \cite{bestvinakapovichkleiner} implies there are obstructions to doing at least one of these three steps integrally. It would be interesting to understand what these obstructions are. 
\end{remark}

\section{\label{davissection}The reflection group method}
The following construction for building closed manifolds out of compact manifolds-with-boundary is known as the Davis reflection group method (see Chapter 11 of \cite{davisbook}).

Let $(M,\partial)$ be a compact $n$-manifold-with-boundary, $L$ a flag triangulation of the boundary and $W$ the corresponding right angled Coxeter group 
$$
W=\left<s_v, v \mbox{ a vertex of }L\mid s_v^2=1\mbox{ and } s_vs_w=s_ws_v \mbox{ if } v \mbox{ is adjacent to } w\right>.
$$
For each vertex $v\in L^{(0)}$, let $St_v$ be the closed star of $v$ in the barycentric subdivision of $L$. Then the sets $\{St_v\}$ form a mirror structure and one constructs a space 
$$
\mathcal U(W,M):=M\times W/\sim
$$ with $(x,u)\sim(y,v)$ if $x=y$ and $uv^{-1}\in W_x$ (where $W_x=\left<s_v\mid x\in St_v\right>$).
\begin{theorem}[Davis, 11.1 in \cite{davisbook}]
Let $\mathcal U:=\mathcal U(W,M)$. Then,
\begin{itemize}
\item
$\mathcal U$ is an $n$-manifold with $W$-action and $\mathcal U/W=M$.
\item
$W$ contains a finite index torsionfree subgroup\footnote{For instance, the commutator subgroup $[W,W]$ of $W$ is torsionfree and of index $2^{|L^{(0)}|}$.} $G$. For any such subgroup, the quotient $\mathcal U/G$ is a closed manifold.
\end{itemize}
Moreover,
\begin{itemize}
\item
if $M$ is aspherical then $\mathcal U$ is also aspherical. 
\end{itemize}
\end{theorem} 
\begin{remark}
If $M$ is aspherical, then its double along the boundary $M\cup_\partial M$ is closed but often not aspherical. One can think of the construction $\mathcal U$ as a ``partial doubling'' along pieces of the boundary which does produce an asperical manifold $\mathcal U$ and, after quotienting out by an appropriate group of covering translations $G$, a closed aspherical manifold $\mathcal U/G$. 
\end{remark}
The last bullet point follows by building the universal cover $\widetilde{\mathcal U}$ from copies of $\widetilde M$ via the (infinitely generated) right angled Coxeter group $\widetilde W$ associated to the flag triangulation of the $\pi_1M$-cover $\widetilde L$ of $L$ and using this description to show that if $\widetilde M$ is acyclic the $\widetilde{\mathcal U}$ is also acyclic. More precisely, Theorem 8.1.6 in chapter 8 of \cite{davisbook} says that  
$$
H_*(\widetilde{\mathcal U})\cong\bigoplus_{\sigma\in\widetilde L}H_*(\widetilde M,\widetilde M^{\sigma})\otimes\mathbb Z\widetilde W^{\sigma},
$$
where the sum is over the simplices of $\widetilde L$\footnote{including the empty simplex} and $\widetilde M^{\sigma}=\cup_{v\in\sigma}St_v$ is a union of closed stars $St_v$ in the barycentric subdivision of $\widetilde L$.\footnote{and $\widetilde W^{\sigma}$ is a subgroup of $\widetilde W$ associated to $\sigma$ that we will not need to know about.}
Since $\widetilde M$ is acyclic and $\widetilde L$ is a flag complex, one concludes that $\overline H_*(\widetilde{\mathcal U})=0$. The same argument works rationally if one only assumes that $\widetilde M$ is $\mathbb Q$-acyclic.
\begin{itemize}
\item
\label{qacyclic}
If $\widetilde M$ is $\mathbb Q$-acyclic then $\widetilde{\mathcal U}$ is $\mathbb Q$-acyclic. 
\end{itemize}
\section{\label{l2betti} $L^2$-Betti numbers and vanishing conjectures}
\subsection{$L^2$-Betti numbers}
Let $X$ be a finite complex with universal cover $\widetilde X$ and fundamental group $\Gamma$. Let $(C^{(2)}_*(\widetilde X;\mathbb R),\partial^{(2)})$ be the complex of square-summable chains on the universal cover. The homology $H^{(2)}_*(\widetilde X)$ of this complex is often an infinite dimensional vector space, but it is also a $\Gamma$-module, and has finite (von Neumann) $\Gamma$-dimension. The $k$-th $L^2$-Betti number of $X$ is the $\Gamma$-dimension of the $k$-th $L^2$-homology module
$$
b_k^{(2)}(X):=\dim_{\Gamma}H^{(2)}_k(\widetilde X).
$$
The $L^2$-Betti numbers have many nice properties (see \cite{lueckbook}). For instance,
\begin{itemize}
\item(Poincare duality)
If $N$ is a closed $n$-manifold then $b^{(2)}_k(N)=b^{(2)}_{n-k}(N)$,
\item(Euler characteristic formula)
$\chi(X)=\sum(-1)^kb_k^{(2)}(X)$, 
\item($\mathbb Q$-invariance)
If $X\ra Y$ is a $\pi_1$-isomorphism and the map on universal covers $\widetilde X\ra\widetilde Y$ is a $\mathbb Q$-homology isomorphism, then $b^{(2)}_k(X)=b^{(2)}_k(Y)$,
\item(Multiplicativity in covers) If $X\ra Y$ is a degree $d$ cover then $b^{(2)}_k(X)=d\cdot b^{(2)}_k(Y)$.
\item(Disjoint unions/induction principle\footnote{We call this the induction principle because it is a consequence on the level of $\dim_{\Gamma}$ of the induction principle for $L^2$-homology stated in Lemma 3.1 of \cite{okunschreve}.}) Suppose $\Gamma$ acts cocompactly by covering translations on a simply connected\footnote{but possibly not connected} space $\hat X$ and the quotient decomposes as a disjoint union of $r$ connected components $\hat X/\Gamma=X_1\coprod\dots\coprod X_r$. Then \begin{equation}
\label{induction}
\dim_{\Gamma}H^{(2)}_k(\hat X)=b_k^{(2)}(X_1)+\dots+b_k^{(2)}(X_r).
\end{equation}
\end{itemize}

\subsection{\label{vanishingconjectures}$L^2$-vanishing conjectures}
A classical conjecture \cite{dodziuk} (see chapter 11 of \cite{lueckbook}), says that for a closed aspherical $n$-manifold the only non-vanishing $L^2$-Betti numbers appear in the middle dimension.\footnote{So, if the manifold is odd dimensional, then all of its $L^2$-Betti numbers should vanish.}
\begin{conjecture}[Singer]
\label{singer}
If $N$ is a closed aspherical $n$-manifold then 
$$
b^{(2)}_k(N)=\left\{\begin{array}{ccc} (-1)^k\chi(N)& if &k=n/2,\\
0&& \mbox{ else}.\end{array}\right.
$$
\end{conjecture}
More recently, it was conjectured that the $L^2$-Betti numbers of any aspherical manifold should vanish above the middle dimension \cite{davisokun}.
\begin{conjecture}[Davis-Okun]
\label{davisokun}
If $M$ is an aspherical $n$-manifold then $b^{(2)}_{>n/2}(M)=0$.
\end{conjecture} 
Even more recently, the paper \cite{okunschreve} focused attention on a restricted version of this and related it to the Singer conjecture.
\begin{conjecture}[Okun-Schreve]
\label{os}
If $(M,\partial)$ is a compact aspherical $n$-manifold-with-boundary then $b^{(2)}_{>n/2}(M)=0$.
\end{conjecture}
This is a generalization the Singer conjecture because if $M$ is closed ($\partial=\emptyset$) then Poincare duality implies that the $L^2$-Betti numbers also vanish below the middle dimension. Using the reflection group method, Okun and Schreve show in \cite{okunschreve} that the Singer conjecture in dimension $\leq n$ implies their conjecture in dimension $\leq n$.

\begin{theorem}[4.5 in \cite{okunschreve}]
\label{okunschrevetheorem}
Conjectures \ref{singer} and \ref{os} are equivalent.
\end{theorem}

Thus, to find counterexamples to the Singer conjecture one can try to build low-dimensional {\it manifold} models for finite classifying spaces $B\Gamma$. In Theorem \ref{rationalmodels} we constructed rational analogues of such low dimensional manifold models. In the rest of this paper, we will explain how to mimic the argument of \cite{okunschreve} in order to get a counterexample to a rational analogue of Singer's conjecture (Corollary \ref{noqsinger}). First, we recall the argument. 

\subsection{Sketch of proof of Theorem \ref{okunschrevetheorem}}
Suppose we know Conjecture \ref{singer} in dimension $\leq n$ and the Conjecture \ref{os} in dimension $<n$. 
The method of Okun and Schreve is to start with an aspherical $n$-manifold with boundary $(M,\partial)$, perform the reflection group construction $\mathcal U/G:=\mathcal U(W,M)/G$ and then keep removing $\pi_1(\mathcal U/G)$-orbits of walls from $\widetilde{\mathcal U}$ and check that the high dimensional $L^2$-homology groups do not increase in the process. At each stage, an orbit of walls is a disjoint union of contractible $(n-1)$-manifolds (possibly with boundary) on which $\pi_1(\mathcal U/G)$ acts by covering translations. The quotient of this orbit is a finite union of aspherical $(n-1)$-manifolds (with boundary). Conjecture \ref{os} for these $(n-1)$-manifolds implies\footnote{Via an appropriate long exact sequence in $L^2$-homology.} that removing them does not increase high dimensional $L^2$-homology. Once all the walls have been removed we are left with $(\mathcal U\setminus Walls)/G=\coprod_{|W/G|} M$ having no more high dimensional $L^2$-homology\footnote{See Theorem \ref{davismanifoldbetti} and (\ref{remove}) below for the precise statement.} than $\mathcal U/G$. From this, one deduces that Conjecture \ref{singer} for the closed aspherical $n$-manifold $\mathcal U/G$ implies Conjecture \ref{os} for the aspherical $n$-manifold-with-boundary $(M,\partial)$.

\subsection{The $L^2$-Betti numbers of $\mathcal U/G$}
The argument in \cite{okunschreve}, sketched above, actually shows a bit more. Given the Singer conjecture, it lets one, to a large extent, compute the $L^2$-Betti numbers of the closed manifold $\mathcal U(W,M)/G$ from those of $M$, even if $M$ is not aspherical.
\begin{theorem}
\label{davismanifoldbetti}
If the Singer conjecture is true for manifolds of dimension $<n$, then 
$$
b^{(2)}_k(\mathcal U/G)=|W/G|b^{(2)}_k(M)
$$
for $k>{\lfloor(n+1)/2\rfloor}$
and\footnote{Here $\lfloor\cdot\rfloor$ is the floor function, so $\lfloor(n+1)/2\rfloor$ is $n/2$ if $n$ is even and $(n+1)/2$ if $n$ is odd.}
$$
b^{(2)}_{\lfloor(n+1)/2\rfloor}(\mathcal U/G)\geq |W/G|b^{(2)}_{\lfloor(n+1)/2\rfloor}(M).
$$
\end{theorem}
Before explaining the proof, we recall the definition of a tidy arrangement.

\begin{definition}[page 3 of \cite{okunschreve}]
Let $N$ be a $\Gamma$-manifold (possibly with boundary) and $\mathcal E=\{E_i\}_{i=0}^r$ a collection of codimension one $\Gamma$-submanifolds\footnote{possibly disconnected} (also possibly with boundary). The arrangement $(N,\mathcal E)$ is called {\it tidy} if 
\begin{enumerate}
\item
$N$ is contractible,
\item
the components of any intersection of the $E_i's$ are contractible,
\item
$\partial N\cap E_i=\partial E_i$ for each $i$, and
\item
$(N,\mathcal E)$ looks locally like a real hyperplane arrangement.
\end{enumerate}
\end{definition}
\begin{proof}
The universal cover $\widetilde{\mathcal U}$ together with the collection of orbits of walls in $\widetilde{\mathcal U}$ under the fundamental group $\pi_1(\mathcal U/G)=\pi_1\mathcal U\rtimes G$ satisfies the last three conditions of a tidy arrangement but not the first one. However, we note that the argument of Okun and Schreve works exactly the same way without the assumption that $\widetilde{\mathcal U}$ is contractible. We only need to know that it is simply connected. The conclusion of their argument is that removing the walls from $\widetilde{\mathcal U}$ does not affect $L^2$-homology in high dimensions, and does not increase $L^2$-homology just above the middle dimension. More precisely
\begin{eqnarray}
\label{remove}
H^{(2)}_{k}(\widetilde{\mathcal U}\setminus Walls)&\ra&H^{(2)}_{k}(\widetilde{\mathcal U})
\end{eqnarray}  
is an isomorphism for $k>\lfloor(n+1)/2\rfloor$ and injective for $k=\lfloor(n+1)/2\rfloor$. Removing the projections of the walls from $\mathcal U$ leaves a disjoint union of $M$ parametrized by $W$ so that $(\widetilde{\mathcal U}\setminus Walls)/(\pi_1U\rtimes G)\cong(\mathcal U\setminus Walls)/G\cong\coprod_{|W/G|}M$. Combining this decomposition, the induction principle (\ref{induction}) and the result that (\ref{remove}) is an isomorphism (injective) in appropriate dimensions gives the equation (inequality) in the statement of the theorem. 

\end{proof}
\begin{remark}
Poincare duality in $L^2$-homology implies that $b^{(2)}_k(\mathcal U/G)=b^{(2)}_{n-k}(\mathcal U/G)$, so this computes the $L^2$-Betti numbers except in dimensions $n/2$ (if $n$ is even) or $(n\pm 1)/2$ (if $n$ is odd).
\end{remark}

\begin{corollary}
\label{noqsinger}
There is a closed manifold with rationally acyclic universal cover and $L^2$-Betti numbers not concentrated in the middle dimension. Moreover, the smallest example of such a manifold has dimension $\leq 7$. 
\end{corollary}
\begin{proof}If the Singer conjecture is false for some manifold of dimension $<7$, then we are done.

So, suppose it is true in dimension $<7$. By Theorem \ref{rationalmodels} we can build a $7$-dimensional compact manifold-with-boundary $(M,\partial)$ so that $\widetilde M$ is rationally acyclic and $\pi_1M=F_2^4$. The classifying map $M\ra BF_2^4$ is a rational equivalence, so the manifold $M$ has non-vanishing $4$-th $L^2$-Betti number $b_4^{(2)}(M)=b_4^{(2)}(F_2^4)=1$. By Theorem \ref{davismanifoldbetti} the closed $7$-manifold $\mathcal U(W,M)/G$ has non-vanishing $b_4^{(2)}$ and by Section \ref{qacyclic} the universal cover $\mathcal U(W,M)$ of this manifold is $\mathbb Q$-acyclic. 
\end{proof}

\subsection{On the rational analogue of the Euler characteristic conjecture}
The Corollary above gives counterexamples to a rational analogue of the Singer conjecture. The Singer conjecture itself was suggested as a way of establishing a more classical conjecture on the sign of the Euler characteristic of a closed aspherical manifold. (See \cite{davisonhopf} for a historical discussion.)
\begin{conjecture}[Euler characteristic conjecture]
If $N$ is a closed aspherical $2n$-manifold then $(-1)^n\chi(N)\geq 0$. 
\end{conjecture}
\begin{remark}
For closed nonpositively curved manifolds, it is usually attributed to H. Hopf. Thurston suggested that it might be true for all closed aspherical manifolds. 
\end{remark}
Are there counterexamples to the rational analogue of the Euler characteristic conjecture?
\begin{question}
Is there a closed $2n$-manifold $N$ with rationally acyclic universal cover $\widetilde N$ and $(-1)^n\chi(N)<0$?
\end{question}
One approach would be to use Theorem \ref{rationalmodels} to build an $8$-manifold-with-boundary $(M,\partial)$ with fundamental group $F_2^5$ and rationally acyclic universal cover. Associated to a flag triangulation of the boundary is the closed $8$-manifold $\mathcal U(W,M)/G$. If we assume the integral Singer conjecture, then Theorem \ref{davismanifoldbetti} implies the Euler characteristic of this manifold is $\chi=b_{4}^{(2)}(\mathcal U/G)-2|W/G|$.\footnote{Note that the theorem does not give any restrictions on $b_4^{(2)}(\mathcal U/G)$.} Is there a flag triangulation of the boundary $\partial$ for which $\chi<0$?

\section{Appendix: Rationally splitting the Hurewicz map in some special cases.}
The goal of this appendix is to indicate how to obtain a special case (Theorem \ref{specialcase}) of the results in section \ref{rationalhomotopysection} using slightly more ``pedestrian'' rational homotopy theory. 
\subsection{Classical rational homotopy theory}
A general reference for everything in this subsection is \cite{felixhalperinthomas}.
The main result we will need is
\begin{theorem}
\label{freeliealgebra}
\begin{equation}
\label{loop}
\pi_{*}(\vee S^r)\otimes\mathbb Q\cong \pi_{*-1}(\Omega\vee S^r)\otimes\mathbb Q
\end{equation}
is the free graded Lie algebra on the vector space $\pi_{r-1}(\Omega\vee S^r)\otimes\mathbb Q$ in dimension $r-1$, 
and the isomorphism (\ref{loop}) sends Whitehead products\footnote{See definition below.} to $(\pm)$ Lie brackets.
\end{theorem}
This result follows from a tensor algebra description of the rational homology of the loop space $H_*(\Omega\vee S^r;\mathbb Q)$ and the identification of the rational homotopy groups as the primitive elements in this tensor algebra.
\subsubsection{Diagonals and primitive elements} An element $x\in H_*(X;\mathbb Q)$ is {\it primitive} if the diagonal
$$
\Delta:H_*(X;\mathbb Q)\ra H_*(X\times X;\mathbb Q)\cong H_*(X;\mathbb Q)\otimes H_*(X;\mathbb Q)
$$ sends $\Delta(x)=x\otimes 1+1\otimes x$. Let $P_*(X;\mathbb Q)\subset H_*(X;\mathbb Q)$ be the subspace of primitive elements. Any map $X\ra Y$ sends primitive elements to primitive elements. Since $H_r(S^r;\mathbb Q)$ is primitive, it follows that the image of the rational Hurewicz map consists of primitive elements. 
\subsubsection{Primitive elements in a loop space}
For loop spaces, more is true (16.10 of \cite{felixhalperinthomas}):
\begin{theorem}[Cartan-Serre]
The Hurewicz homomorphism gives an isomorphism
\begin{equation}
\pi_*(\Omega X)\otimes\mathbb Q\cong P_*(\Omega X;\mathbb Q).
\end{equation}
\end{theorem}
\subsubsection{Lie algebras structures on homotopy groups}
There are two ways to make the rational homotopy groups into a Lie algebra. 
\begin{itemize}
\item
For any $f\in\pi_k(X),g\in\pi_n(X)$, composing the attaching map $S^{k+n-1}\ra S^k\vee S^n$ of the top dimensional cell in $S^k\times S^n$ with the wedge map $f\vee g:S^k\vee S^n\ra X$ gives the {\it Whitehead product} 
\begin{eqnarray}
\pi_k(X)\times\pi_n(X)&\ra&\pi_{k+n-1}(X),\\
(f,g)&\mapsto&[f,g]_W.
\end{eqnarray} 
It turns out that this satisfies the Lie bracket identities rationally. 
\item
Concatenating loops in the loop space $\Omega X$ gives a map $\Omega X\times\Omega X\ra\Omega X$, a multiplication
\begin{equation}
\label{multiplication}
H_*(\Omega X;\mathbb Q)\otimes H_*(\Omega X;\mathbb Q)\cong H_*(\Omega X\times\Omega X;\mathbb Q)\ra H_*(\Omega X;\mathbb Q)
\end{equation} 
and from this a {\it Lie bracket} $[\cdot,\cdot]$ on $H_*(\Omega X;\mathbb Q)$. This Lie bracket preserves the subspace of primitive elements $P_*(\Omega X;\mathbb Q)\cong \pi_*(\Omega X)\otimes\mathbb Q$.
\end{itemize}
Whitehead products and Lie brackets are related by the map
$$
\theta:\pi_*(X)\cong\pi_{*-1}(\Omega X)\ra H_{*-1}(\Omega X;\mathbb Q).
$$  
\begin{proposition}[16.11 of \cite{felixhalperinthomas}]
\label{products}
For $f\in\pi_k(X),g\in\pi_n(X)$ we have 
$$
\theta[f,g]_W=(-1)^{k}[\theta f,\theta g].
$$
\end{proposition}
\subsubsection{Homology of the loop space of a suspension}
The rational homology of the loop space of a suspension\footnote{Here we pick a basepoint $y\in Y$ and the suspension is the {\it reduced suspection} $\Sigma Y= Y\times I/(Y\times \{0,1\}\cup\{y\}\times I)$} has the following description as a tensor algebra\footnote{with multipication given by (\ref{multiplication})} (\cite{bottsamelson}, see also 4J.1 in \cite{hatcher} and Corollary 4.1.5 in \cite{neisendorfer} for the last bit).
\begin{theorem}[Bott-Samelson]For a connected space $Y$
$$
H_*(\Omega\Sigma Y;\mathbb Q)\cong T(\overline H_*(Y;\mathbb Q))
$$ 
with the diagonal map on $H_*(\Omega\Sigma Y;\mathbb Q)$ determined by the diagonal map on $H_*(Y;\mathbb Q)$.
\end{theorem} 
\begin{remark}If $Y=\vee S^{r-1}$ then the diagonal map on $H_*(Y;\mathbb Q)$ has to be trivial.\footnote{Given by $\Delta(y)=y\otimes 1+1\otimes y$ because there are no homology classes in dimensions between $0$ and $r-1$.} It follows, for such $Y$, that the primitive elements of $H_*(\Omega\Sigma Y;\mathbb Q)$ are precisely the free graded Lie algebra $L(\overline H_*(Y;\mathbb Q))$ inside the tensor algebra (see Chapter 21 of \cite{felixhalperinthomas}).
\end{remark}

\begin{proof}[Proof of Theorem \ref{freeliealgebra}]
We combine these for $Y=\vee S^{r-1}, X=\Sigma Y=\vee S^r$ to prove the theorem. 
The Cartan-Serre theorem identifies the rational homotopy groups $\pi_*(\Omega\vee S^r)\otimes\mathbb Q$ as the primitive elements in $H_*(\Omega\vee S^r;\mathbb Q)$ and the Bott-Samelson theorem (together with the remark following it) implies the primitive elements in $H_*(\Omega\vee S^r;\mathbb Q)=H_*(\Omega\Sigma\vee S^{r-1};\mathbb Q)$ are the free graded Lie algebra on $H_{r-1}(\vee S^{r-1};\mathbb Q)\cong H_r(\vee S^r;\mathbb Q)\cong\pi_r(\vee S^r)\otimes\mathbb Q\cong \pi_{r-1}(\Omega\vee S^r)\otimes\mathbb Q$. Finally, Proposition \ref{products} relates Whitehead products with Lie brackets.
\end{proof}

\subsection{A special case}
In the rest of this appendix we will prove the following special case of Theorem \ref{rationalmoorespace}.
\begin{theorem}
\label{specialcase}
For $d\geq 4$, $\Gamma=F_2^d$ and $r>d/2$ there is a $\Gamma$-complex $V$ with rational homology $\overline H_*(V;\mathbb Q)$ concentrated in dimension $r$ and equal to $D\otimes\mathbb Q$.  
\end{theorem}
\begin{remark}
We show this for $d-1>r>d/2$. It follows for all larger $r$ by taking suspensions.
\end{remark}

Let $D$ be the dualizing module of the group $\Gamma$ and $0\ra F_{d+r}\ra\dots\ra F_r\ra D\ra 0$ the resolution given by $F_{k}=Hom_{\Gamma}(C_{d+r-k}(E\Gamma);\mathbb Z\Gamma)=C_{d+r-k}(E\Gamma)^*$. Build an $(r+1)$-dimensional $\Gamma$-complex $V^{r+1}=F_r\cup F_{r+1}$ with $F_r$ represented by a wedge of spheres and $F_{r+1}$ attached so that $H_r(V^{r+1})\cong D$. Identify $D$ as an abelian group with the homology of a wedge of $r$-spheres $H_r(\vee S^r)$  and build a homology isomorphism $f_{r+1}:V^{r+1}\ra\vee S^r$. 
This map is rationally $(r+1)$-connected.\footnote{It is rationally onto on $\pi_*$ because $\pi_*(\vee S^r)\otimes\mathbb Q$ is generated by Whitehead products of $\pi_r(V^r)\cong\pi_r(\vee S^r)$.} We construct the $\Gamma$-complex $V$ inductively. The induction step is expressed in the following claim, which is proved in the remainder of the appendix.  
\begin{itemize}
\item[\bf{Claim:}]
Given a rationally $k$-connected map $f_k:V^{k}\ra\vee S^r$ we can $\Gamma$-equvariantly attach $(k+1)$-cells and get a $\Gamma$-complex $V^{k+1}=V^k\cup F_{k+1}$ and a rationally $(k+1)$-connected map $f_{k+1}:V^{k+1}\ra\vee S^r$.
\end{itemize} 
\begin{remark}
From now on, everything will be done rationally and we will sometimes omit this from the notation.
\end{remark}
Let $E$ be the homotopy fibre of $V^k\ra\vee S^r$. It is $(r-1)$-connected so $\pi_k(E)=H_k(E)$ and the spectral sequence of the fibration $E\ra V^k\ra\vee S^r$ implies that $H_k(E)\cong H_k(V^k)$. From this it follows that the rational Hurewicz map $\pi_k(V^k)\ra H_k(V^k)$ is onto (see the diagram below). Let $K_k$ be its kernel.
\begin{equation}
\begin{array}{ccccccc}
&&0&&&&\\
&&\downarrow&&&&\\
&&K_k&&&&\\
&&\downarrow&&&&\\
\pi_k(E)&\ra&\pi_k(V^k)&\ra&\pi_k(\vee S^r)&\ra &0\\
||&&\downarrow&&&&\\
H_k(E)&\cong&H_k(V^k)&&&&
\end{array}
\end{equation} 
From the diagram it also follows that the map $\pi_k(E)\ra \pi_k(V^k)$ is injective and that the composite map $K_k\ra\pi_k(\vee S^r)\cong \pi_{k-1}(\Omega\vee S^r)$ is an isomorphism. We describe $K_*$ as a $\mathbb Q\Gamma$-module. 
\begin{lemma}
The kernel $K_k$ is precisely the subspace of Whitehead products in $\pi_k(V^k)$. The map $\theta f_k:\pi_*(V^k)\ra\pi_{*-1}(\Omega\vee S^r)$ identifies it with $\pi_{k-1}(\Omega\vee S^r)$. So, $K_k$ is the degree $k-1$ piece of the free Lie algebra $L(D)$ (on the $\mathbb Q\Gamma$-module $D$ placed in degree $r-1$). 
In particular, the kernel $K_*$ is non-zero only for $*=r+n(r-1),n>0$ and
$K_{2r-1}\cong[D,D]$.
\end{lemma}
\begin{proof}
Let $Wh_*$ be the closure of $\pi_r(V^k)$ under Whitehead products. This $Wh_*$ is a $\mathbb Q\Gamma$-module and a Lie algebra.
The map 
\begin{equation}
\label{loop2}
\theta f_k:\pi_{*}(V^k)\ra\pi_{*-1}(\Omega\vee S^r)
\end{equation}
sends Whitehead products to ($\pm$) Lie brackets and, by Theorem \ref{freeliealgebra}, the codomain $\pi_{*-1}(\Omega\vee S^r)$ is the free graded Lie algebra on $D\cong\pi_r(V^k)\cong\pi_{r-1}(\Omega\vee S^r)$. Consequently $Wh_*$ is also a free graded Lie algebra and the map $\theta f_k$ is an isomorphism $Wh_*\cong L(\pi_{r-1}(\Omega\vee S^r))$.

Any non-trivial Whitehead product is zero in homology, so $Wh_k\subset K_k$.
Since $\theta f_k:K_k\ra\pi_{k-1}(\Omega\vee S^r)$ is an isomorphism we conclude that $K_k$ is precisely the subspace of Whitehead products.
In particular, $K_{2r-1}=[\pi_r(V^k),\pi_r(V^k)]_W\cong[D,D]$  as a $\mathbb Q\Gamma$-module.
\end{proof}

Now we show the Hurewicz map splits rationally. 
\begin{lemma}
The extension of $\mathbb Q\Gamma$-modules 
\begin{equation}
\label{extension}
0\ra K_k\ra\pi_k(V^{k})\ra H_k(V^{k})\ra 0
\end{equation}
has a $\Gamma$-equivariant section $s$.
\end{lemma}
\begin{proof} We show that the Ext group classifying all extensions  vanishes.
\begin{itemize}
\item
The extension is classified by an element of the group $Ext^1_{\Gamma}(H_k(V^{k}),K_k)$. 
\item
To get this $Ext$ group, take the free resolution $\dots\ra F_{k+2}\ra F_{k+1}\ra H_{k}(V^{k})\ra 0$ and compute the homology of the complex $Hom_{\Gamma}(F_{*+{k+1}},K_k)$ in dimension $*=1$. 
\item
Since $F_*$ is $\Gamma$-dual to $C_{d+r-*}(E\Gamma)$ the Ext group gets identified with the group homology 
$$
Ext^1_{\Gamma}(H_k(V^{k}),K_k)\cong
H_{d+r-(k+2)}(B\Gamma;K_k).
$$
\item
For $d-1>r>d/2$ we have $d+r-((2r-1)+2)>0>d+r-((3r-2)+2)$ so the only possible nontrivial extension occurs for $k=2r-1$ and is classified by an element of the group $H_{d-1-r}(B\Gamma;[D,D])$. 
\item 
Next, we show this group is zero by proving $H_{>0}(BF_2^d;[D,D])=0$. Since $D\otimes D$ splits (rationally) as a direct sum $Alt^2D\oplus Sym^2D$ and one of these two is $[D,D]$ it is enough to show that  
\begin{equation}
\label{dualvanish}
H_{>0}(BF_2^d;D\otimes D)=0.\footnote{More generally, one expects for torsionfree lattices $\Gamma$ in locally symmetric spaces of $\mathbb Q$-rank $q$ that $H_{>d-q}(B\Gamma;D\otimes D)=0$. This is true, for instance, for finite index torsionfree subgroups in $\SL_{q+1}\mathbb Z$ \cite{avramidifillings}.}
\end{equation} 
\begin{proof}
This is true for $d=1$ because 
\begin{eqnarray}
H_1(BF_2;D\otimes D)&\cong& H^0(BF_2;D)=D^{F_2}=0,\\
H_{>1}(BF_2;D\otimes D)&\cong& H^{<0}(BF_2;D)=0.
\end{eqnarray}
Denote by $D(\Gamma)$ the dualizing module of the group $\Gamma$. The dualizing module of $F_2^d$ splits as a tensor product of the dualizing modules for the individual $F_2$-factors 
$$
D(F_2^d)=D(F_2)\otimes\dots \otimes D(F_2),
$$ 
so (\ref{dualvanish}) for general $d$ follows from the Kuenneth formula.  
\end{proof}  
\end{itemize}
This completes the proof of Lemma \ref{extension}.
\end{proof}

Since the extension (\ref{extension}) has a $\Gamma$-equivariant section $s:H_k(V^k)\ra\pi_k(V^k)$, we can $\Gamma$-equivariantly attach the generators $F_{k+1}$ of $H_k(V^k)$ to build a $\Gamma$-complex $V^{k+1}=V^k\cup F_{k+1}$. Moreover, after modifying the map $f_k$ on the $k$-skeleton if necessary\footnote{We are allowed to modify $f_k$ by an integral map $F_k\ra K_k$. Since $s(H_k(V^k))$ is a direct summand of $F_k$ and $\im f_k\cong K_k$ we can modify the map $f_k$ so that $s(H_k(V^k))$ is in the kernel. This modified map extends to $f_{k+1}$.} we can extend it to a $(k+1)$-connected map $f_{k+1}:V^{k+1}\ra\vee S^r$. This finishes the inductive step. Proceeding in this way, we get the $\Gamma$-complex $V$. Since its reduced cellular chain complex is $(F_*,Nf_*)$ for some positive integer $N$, it has the correct rational homology. This completes the proof of Theorem \ref{specialcase}.

\bibliographystyle{amsplain}
\bibliography{noqsingerbib}

\end{document}